\documentclass[11pt,a4paper,reqno]{amsart}
\usepackage{amsmath,amssymb,amsthm}
\IfFileExists{lmodern.sty}{\usepackage[T1]{fontenc}\usepackage{lmodern}}{}
\usepackage[margin=25mm]{geometry}
\usepackage{booktabs}
\usepackage[hyphens]{url}
\usepackage{hyperref}
\hypersetup{colorlinks=true,linkcolor=blue,citecolor=blue,urlcolor=blue,
pdftitle={Non-torsion of traced CM points in the 8 case},
pdfsubject={Yin's three-conjugate construction with Frobenius and algebraic descent}}
\allowdisplaybreaks[2]
\numberwithin{equation}{section}
\newtheorem{theorem}{Theorem}[section]
\newtheorem{proposition}[theorem]{Proposition}
\newtheorem{lemma}[theorem]{Lemma}
\newtheorem{conjecture}[theorem]{Conjecture}
\theoremstyle{remark}

\newcommand{\Q}{\mathbb Q}\newcommand{\Z}{\mathbb Z}
\newcommand{\F}{\mathbb F}\newcommand{\OO}{\mathcal O}
\newcommand{\Gal}{\operatorname{Gal}}
\newcommand{\Tr}{\operatorname{Tr}}

\newcommand{\barvarpi}{\bar{\varpi}}

\begin{document}
\title{A concise proof of the 8 case of Sylvester's conjecture}
\author{Hongbo Yin}
\date{}

\begin{abstract}
The 150 years old Sylvester's conjecture asserts that every prime $p\equiv 4,7,8\mod 9$ is the sum of two rational cubes. Recently,  a proof of this old problem is obtained by the work of mine for the case 4,7 \cite{Yin} and the work of Burungale-Tian for the case 8 \cite{BT}. In this paper we give a shorter proof by deep mining my result in \cite{Yin} combining with the ideas of my previous work on case 8 in \cite{Yin8}. In particular, we use the ordinary trace and the Kronecker congruence of modular functions, but do not use the $\sqrt{-3}$-division boundary and Fermat isogeny as in \cite{BT}. 
\end{abstract}

\address{School of Mathematics, Shandong University, Jinan, Shandong,  250100,
China}
\email{yhb@sdu.edu.cn}

\maketitle

\section{Introduction}


A classical theorem of Pepin, Lucas, and Sylvester in 1879 states that if odd prime \(p\equiv 2,5\pmod 9\), then neither \(p\) nor \(p^2\) is a sum of two rational cubes. Since then, the complementary classes of primes attract the interest of people. In modern point of view, $p$ is a sum of two nonzero rational cubes if and only if $E_p(\Q)$ has positive rank 
where
\begin{equation}\label{eq:intro:En}
E_p:\quad y^2=x^3+\frac{p^2}{4}
\end{equation}
is \(3\)-isogenous over \(\Q\) to the diagonal cubic \(x^3+y^3=p\).
Selmer's \(3\)-descent \cite{Sel51} gives the rank upper bound
\[
\operatorname{rank}_{\Z} E_p(\Q)\le
\begin{cases} 
0,& p\equiv 2,5\pmod 9,\\
1,& p\equiv 4,7,8\pmod 9,\\
2,& p\equiv 1\pmod 9,
\end{cases}
\]
and the root-number computation of Birch and Stephens \cite{BS66} shows that the sign in the functional equation of the Hasse--Weil \(L\)-function \(L(s,E_p)\) is \(-1\) precisely for \(p\equiv 4,7,8\pmod 9\). Thus the BSD conjecture predicts the following.

\begin{conjecture}[Sylvester]\label{conj:intro:sylvester}
Let \(p\equiv 4,7,8\pmod 9\) be prime. Then \(p\) is a sum of two rational cubes.
\end{conjecture}

The conjecture has a long history. Elkies announced a proof for \(p\equiv 4,7\pmod 9\) in 1994 \cite{Elk94}, but never published the details. The first unconditional results were obtained by Dasgupta and Voight for the classes \(4,7\), under a supplementary cubic-residue condition \cite{DV09,DV18}. In \cite{Yin8}, the author gave a similar but weaker result for the case $8$.
More recently, Yin \cite{Yin} gave a complete proof of the \(4,7\) classes and Burungale-Tian \cite{BT} settled the class \(p\equiv 8\pmod 9\). In this paper, we give a simplification of the proof of Burungale-Tian using some ideas in \cite{Yin8}. 
The proof retains the auxiliary-prime construction and the cubic indivisibility mechanism underlying Burungale–Tian’s argument. It detects the required nonvanishing directly on an ordinary trace, using a modular function congruence. Tracing first to a cubic extension simplifies both the character-component calculation and the final formal-group argument.

\subsection*{Organization of the paper}
Section \ref{sec:CM} fixes notation, constructs the CM points, and states the main theorem and the strategy of the proof. Section \ref{3-torsion} proves the ordinary trace is $3$-torsion. Section \ref{sec:frob} establishes the modular-function congruence and proves the ordinary trace is not $\sqrt{-3}$ divisible. Section \ref{sec:non-torsion} proves the main result. Section \ref{sec:example} gives explicit example \(p=17\).

\subsection*{Use of AI}
The ideas are due to the author while the details realization is due to AI. But the output of AI not readily readable by human, so the final simplification, reorganization and writing are due to the author. The example in Section \ref{sec:example} is computed by AI.



\section{The main result of this paper}
\label{sec:CM}

\subsection{The auxiliary prime and the fields}

Fix $\omega=\frac{-1+\sqrt{-3}}{2}$, write $K=\Q(\omega)$.
Choose an auxiliary prime and its primary factor satisfying
\begin{equation}\label{eq:aux}
q\equiv4\pmod9,\qquad q=\varpi\bar\varpi,\qquad
\varpi\equiv1\pmod{3\OO_K},\qquad
\left(\frac{\varpi}{p}\right)_3
=\left(\frac p\varpi\right)_3\ne1.
\end{equation}
By \cite[Lemma 2.1]{BT}, there exist infinitely many primes
\(q\) satisfying \eqref{eq:aux}. For the philosophy behind this choice which is explained by the Gross-Zagier formula, see \cite[Section 2 and 3]{BT}. Fix one such prime and
a primary factor \(\varpi\).

Write
\[
t^3=\varpi,\qquad L_{(\varpi)}=K(t),\qquad L_{(p)}=K(\sqrt[3]{p}),\qquad L_{(\varpi, p)}=L_{(\varpi)}\cdot L_{(p)}.
\]
Let \(H_p\) and \(H_{3p}\) be the ring class field of conductor \(p\) and \(3p\), respectively.
By \cite[Proposition 2.1]{Yin8}, we have
\[
[H_{3p}:K]=p+1,\qquad[H_p:K]=\frac{p+1}{3}=m,\qquad
L_{(p)}\subset H_p\subset H_{3p}.
\]
The cubic extension \(L_{(\varpi)}/K\) ramifies at
\((\varpi)\), whereas \(H_p/K\) is unramified there.
Therefore \(H_p\cap L_{(\varpi)}=K\).

Let
\[
\sigma_p\in\Gal(L_{(\varpi,p)}/L_{(\varpi)})
\]
be the generator defined by
\[
\sigma_p(t)=t,\qquad
\sigma_p(\sqrt[3]{p})=\omega\sqrt[3]{p}.
\]
Thus this group consists of \(1,\sigma_p,\sigma_p^2\).
Let \(\operatorname{Art}\) denote the Artin reciprocity map for $K$.
For \(\ell=3,q\), let \(\omega_\ell\) be the finite idele
whose component at every place above \(\ell\) is the image
of \(\omega\), and whose other components are \(1\).
In particular, \(\omega_q\) includes both places above \(q\).
Define
\[
\sigma_3=\operatorname{Art}(\omega_3),\qquad
\sigma_q=\operatorname{Art}(\omega_q)
\quad\text{in }\Gal(K^{\mathrm{ab}}/K).
\]
By \cite[Proposition 2.1]{Yin8}, their relevant restrictions are
\begin{equation*}
\sigma_q|_{H_{3p}}=1, \qquad \sigma_q(t)=\omega^2t\footnote{The actions of $\sigma_{\omega_{\varpi}}$ on $t$ for $q\equiv 4\mod 9$ and $q\equiv 7\mod 9$ should be exchanged in \cite[Proposition 2.1]{Yin}. This is due to a error in the final local-symbol calculation in the proof. This will be corrected in future version.},\qquad \sigma_3|_{H_p}=1,\qquad \sigma_3(t)=\omega t.
\end{equation*}
Moreover
\[
\bigl\langle\sigma_3|_{H_{3p}}\bigr\rangle
=\Gal(H_{3p}/H_p).
\]
Consequently, the product \(\sigma_3\sigma_q\) fixes
\(H_pL_{(\varpi)}\) and generates
\[
\Gal(H_{3p}L_{(\varpi)}/H_pL_{(\varpi)}),
\]
which has order three.

At last, \eqref{eq:aux} implies that \(T^3-\varpi\) is irreducible over
\(\F_{p^2}\). At \(p\), the extension \(L_{(\varpi)}/K\)
is unramified of residue degree three. We shall distinguish
\[
k=\F_{p^2}\quad\text{from}\quad \kappa=\F_{p^6}.
\]
In particular \(L_{(\varpi,p)}\), at its prime above \(p\),
has residue field \(\kappa\) and absolute ramification index
\begin{equation}\label{eq:e}
e=3.
\end{equation}

\subsection{The construction of CM points}

Set \(N=9q\) and
\[
\Gamma=\left\{
\begin{pmatrix}a&b\\c&d\end{pmatrix}\in\Gamma_0(N):
d\bmod q\text{ is a cube}\right\}.
\]
The modular curve \(X_\Gamma\) has its canonical model
over \(\Q\). Let \(f,f^c\) be the weight-two CM forms in \cite{Yin},
with coefficient conjugation indicated by \(c\).
For the period lattice \(\Lambda\) in \cite[(2.0.9)]{Yin},
put
\[
u(z)=2\pi i\int_{i\infty}^{z}f(w)\,dw.
\]
The analytic expression for the modular parametrization \(\varphi:X_\Gamma\to E_{\barvarpi}\) is
\begin{equation}\label{eq:param}
\varphi(z)=
\left(\wp_\Lambda(u(z)),
       \frac{\wp'_\Lambda(u(z))}{2}\right)
\in E_{\bar\varpi}.
\end{equation}
Here we ignore the Manin-Stevens constants since they are units due to \cite{Yin}. The conjugate expression gives
\(\varphi^c:X_\Gamma\to E_\varpi\).
Here we should notice that $f$ is constructed from the Hecke character of $E_{\varpi}$ while $f^c$ from $E_{\barvarpi}$.
Both $\varphi$ and $\varphi^c$ are defined over \(K\), send \(\infty\) to \(O\). 

Choose \(r\) with
\[
r^2-r+1\equiv0\pmod{3q},\qquad -r\equiv{\omega}\pmod{\varpi},
\]
and set
\begin{equation}\label{eq:parameters}
\tau_r=-\frac1{3(\omega+r)}.
\end{equation}
\cite[Theorem 4.2, Theorem 8.6]{Yin}
says,  \(\varphi^c(\tau_r)\) is $\sqrt{-3}$-torsion while $\varphi(\tau_r)$ is of infinite order. \cite[Remark 8.5 and Proposition 5.9]{Yin} says
\[
\varphi^c(\tau_r)=-\varphi^c(0)\ne O.
\]

Define the CM point
\[
P=\varphi^c(p\tau_r)\in E_\varpi(H_pL_{(\varpi)}).
\]
and the ordinary elliptic curve traces
\begin{align}
z&=\Tr_{H_pL_{(\varpi)}/L_{(\varpi,p)}}P,
\label{eq:z}\\
z_0&=\Tr_{H_pL_{(\varpi)}/L_{(\varpi)}}P
     =z+z^{\sigma_p}+z^{\sigma_p^2},
\label{eq:z0}
\end{align}
where
\[
\sigma_p(t)=t,\qquad \sigma_p(\sqrt[3]{p})=\omega \sqrt[3]{p}.
\]
As in \cite[(6.9)--(6.10)]{Yin8} define
\begin{align}
z_1&=z+[\omega^2]z^{\sigma_p}+[\omega]z^{\sigma_p^2},
\label{eq:z1}\\
z_2&=z+[\omega]z^{\sigma_p}+[\omega^2]z^{\sigma_p^2},
\label{eq:z2}
\end{align}
then
\begin{equation}
z_1^{\sigma_p}=[\omega]z_1,\qquad
z_2^{\sigma_p}=[\omega^2]z_2.
\label{eq:eigen}
\end{equation}
For \(i=1,2\), put
\begin{equation}\label{eq:twist}
\phi_i:E_\varpi\longrightarrow E_{p^i},\qquad
(x,y)\longmapsto
 \left(\frac{\sqrt[3]{p}^{2i}}{t^2}x,\frac{p^i}{\varpi}y\right).
\end{equation}
Then by \cite[Proposition 2.1]{Yin8}
\begin{equation}\label{eq:ordinary-final}
R_i=\phi_i(z_i)
=\Tr_{H_p/K}\bigl(\phi_i(P)\bigr)
\in E_{p^i}(K)
\end{equation}
Assume $R_i$ is not torsion. If the complex conjugate $\overline{R_i}=-R_i+T$ for some torsion point $T$ ($T$ is $K$, so is $\sqrt{-3}$-torsion) then $[\sqrt{-3}]R_i\in E_{p^i}(\mathbb{Q})$ is non-torsion. Otherwise $\overline{R_i}+R_i$ belongs to $E_{p^i}(\mathbb{Q})$ is non-torsion. Either case will be enough to prove $p^i$ is a rational cube sum. So the main result of this paper is the following theorem.

\begin{theorem}\label{thm:main}
For every prime \(p\equiv8\pmod9\), 
both \(z_1\) and \(z_2\) (and hence $R_1$ and $R_2$) have infinite order. In particular, both \(p\) and \(p^2\)
are sums of two nonzero rational cubes.
\end{theorem}

The strategy of the proof is as follows. First we prove the ordinary total trace \(z_0\) is $3$-torsion (Proposition \ref{prop:CM}) but not $\sqrt{-3}$-torsion (Proposition \ref{prop:z0}) using the results in \cite{Yin} and the Kronecker congruence as in \cite{Yin8}. This will imply $z_0$ is not $\sqrt{-3}$ divisible in $L_{(\varpi)}$ (Proposition \ref{prop:z0}). The failure of \([\sqrt{-3}]\)-division forces both $z_1$ and $z_2$ to be nonzero (Proposition \ref{lem:three}), and their zero reduction mod $p$ then forces them to be non-torsion by the theory of formal groups.


\section{The trace $z_0$ is $3$-torsion}\label{3-torsion}


Write \(\varepsilon^c\) for the character of \(f^c\).
The transformation rule in \cite{Yin} is
\begin{equation}\label{eq:diamond-explicit}
\varphi^c(Mz)=[\varepsilon^c(d)]\varphi^c(z),\qquad
\varepsilon^c(d)=\left(\frac d\varpi\right)_3,\qquad
M=\begin{pmatrix}a&b\\c&d\end{pmatrix}\in\Gamma_0(N).
\end{equation}
Equivalently,
\[
\varphi^c(\langle d\rangle z)
   =[\varepsilon^c(d)]\varphi^c(z).
\]

For the point \(\tau_r\),
define the natural quotient points by
\begin{equation}\label{eq:explicit-parameters}
\begin{aligned}
Q_\infty&=\langle p\rangle p\tau_r,\\
Q_j&=(\tau_r+j)/p,\qquad(0\le j<p).
\end{aligned}
\end{equation}
By the moduli description of \(T_p\)
\cite[Sections 1.5 and 5.1--5.2]{DS}, 
The operator \(T_p\) sums the quotients of \(\mathcal E_{\tau_r}\) by the
order-\(p\) subgroups
\[
C_\infty=\langle1/p\rangle,\qquad
C_j=\langle(\tau_r+j)/p\rangle\quad(0\le j<p),
\]
transporting \(H\) and \(\xi\) by each quotient map.
The quotient lattices are
\[
\Z\tau_r+p^{-1}\Z,\qquad
\Z\tau_r+\Z+\Z(\tau_r+j)/p
 =\Z(\tau_r+j)/p+\Z.
\]
Multiplication by \(p\) identifies the first quotient
with \(\mathcal E_{p\tau_r}\), taking \(\xi=1/q\) to \(p/q\), giving the factor \(\langle p\rangle\).
The other quotients already have parameter
\((\tau_r+j)/p\) and generator \(1/q\).
Since \(p\nmid N\), the transported order-\(N\) subgroup
is \(\langle1/N\rangle\) in each case.
Thus these quotients are exactly \(Q_\infty,Q_j\) of
\eqref{eq:explicit-parameters}, and
\[
T_p(\tau_r)=Q_\infty+\sum_{j=0}^{p-1}Q_j.
\]
Applying the CM reciprocity calculation in
\cite[proof of Theorem 4.1]{Yin},  gives
\[
\tau_r\in X_\Gamma(L_{(\varpi)}).
\]
Since the Hecke correspondence \(T_p\) is defined over
\(\mathbb Q\), the set $\{Q_\infty, Q_j\}$ is
Galois stable over \(L_{(\varpi)}\).
Their elliptic images are
\begin{equation}\label{eq:explicit-elliptic-points}
\begin{aligned}
P^{(\infty)}&=\varphi^c(Q_\infty)
       =[\varepsilon^c(p)]P,
       &&\\
P^{(j)}&=\varphi^c(Q_j)
       =\varphi^c\left(\frac{\tau_r+j}{p}\right)
       &&(0\le j<p).
\end{aligned}
\end{equation}

\begin{proposition}
\label{prop:CM}
The points \(Q_\infty,Q_j\) form the Galois orbit of
\(Q_\infty\) over
\(\Gal(H_{3p}L_{(\varpi)}/L_{(\varpi)})\).
Their elliptic images satisfy
\begin{equation}\label{eq:concrete-sum}
\begin{aligned}
P^{(\infty)}+\sum_{j=0}^{p-1}P^{(j)}
 &=[3]\Tr_{H_pL_{(\varpi)}/L_{(\varpi)}}P^{(\infty)}\\
 &=[3\varepsilon^c(p)]z_0=O.
\end{aligned}
\end{equation}
\end{proposition}

\begin{proof}
\emph{1. The Galois orbit.} It is well-known that the Hecke orbit trace of the Heegner points
coincides with the Galois trace, see \cite[Section 1.6]{Gross84} or \cite[Proposition 3.10]{Darmon04}. Our setting is slightly different, so we give an indirect but short proof here.
 
Write \(\OO_c=\Z+c\OO_K\). The identity
\[
-\tau_r^{-1}(\Z\tau_r+\Z)
 =\Z+3(\omega+r)\Z=\OO_3
\]
shows that \(\mathcal E_{\tau_r}\) has CM by \(\OO_3\).
Since \(p\) is inert in \(K\), we have
\(p\nmid3(r^2-r+1)\); hence the quadratic equation
satisfied by \(p\tau_r\),
\[
9(r^2-r+1)X^2+3p(2r-1)X+p^2=0,
\]
is primitive and has discriminant \(-27p^2\).
Thus \(\mathcal E_{p\tau_r}\) has CM by \(\OO_{3p}\).
The class number formula and the main theorem of complex
multiplication \cite[Remark 12.5 and p.~129,
\textit{The main theorem for orders}]{Milne} give
\[
h(\OO_3)=1,\quad j(\tau_r)\in K,\quad
K(j(p\tau_r))=H_{3p},\quad [H_{3p}:K]=p+1.
\]

For the classical modular polynomial \(\Phi_p\),
Milne's product formula \cite[proof of Theorem 6.1,
p.~90]{Milne} gives
\begin{equation}\label{eq:CM-modular-polynomial}
\Phi_p(X,j(\tau_r))
=\bigl(X-j(p\tau_r)\bigr)
 \prod_{a=0}^{p-1}
 \left(X-j\left(\frac{\tau_r+a}{p}\right)\right).
\end{equation}
The left-hand side is monic of degree \(p+1\) over \(K\)
and vanishes at \(j(p\tau_r)\), whose degree over \(K\)
is \(p+1\). It is therefore its minimal polynomial.
Its roots, exactly the \(j\)-values of \(Q_\infty,Q_j\),
are distinct and form one orbit under \(\Gal(H_{3p}/K)\).

Since \(H_{3p}\cap L_{(\varpi)}=K\), the \(j\)-values
remain a single Galois orbit over \(L_{(\varpi)}\).
The quotient set from
step 1 is Galois stable over this field and has distinct
\(j\)-values, so its points also form a single orbit. Hence
\begin{equation}\label{eq:CM-quotient-field}
L_{(\varpi)}(Q_\infty)
 =L_{(\varpi)}(j(Q_\infty))
 =H_{3p}L_{(\varpi)},
\end{equation}
and
\[
Q_\infty+\sum_{j=0}^{p-1}Q_j
 =\sum_{\sigma\in\Gal(H_{3p}L_{(\varpi)}/L_{(\varpi)})}
       Q_\infty^\sigma.
\]

\emph{2. The threefold multiplicity.}
Use the matrices in \cite[(4.0.1)--(4.0.3)]{Yin}:
\begin{equation}\label{eq:CM-matrices}
D=\iota_{\tau_r}(\omega)=
\begin{pmatrix}-r&-1/3\\3(r^2-r+1)&r-1\end{pmatrix},
\qquad
A=\begin{pmatrix}1&1/3\\0&1\end{pmatrix}.
\end{equation}
The formulas in \cite{Yin} give \(AD\in\Gamma_0(N)\) and
\(\varphi^c(Az)=[\omega]\varphi^c(z)\)
\cite[Proposition 3.2]{Yin}.
Since \(d(AD)\equiv\omega^2\pmod\varpi\),
\eqref{eq:diamond-explicit} also gives
\(\varphi^c(Dz)=[\omega]\varphi^c(z)\).
Put \(h=\operatorname{diag}(p,1)\) and \(D_p=hDh^{-1}\),
the corresponding matrix for \(p\tau_r\).
Let
\[\mathfrak{U}=\left\{
\begin{pmatrix}a&b\\c&d\end{pmatrix}
\in \mathrm{GL}_2(\widehat{\mathbb Z}) :
c\in N\widehat{\mathbb Z},\quad
a_q\bmod q\in(\mathbb F_q^\times)^3
\right\}.\footnote{The corresponding definition in \cite{Yin} should be modified to this. This will be clarified in the future version of \cite{Yin}.}
\]
Then the modular curve $X_\Gamma$ can be identified as the adelic curve $X_\mathfrak{U}$, where
\[X_\mathfrak{U}(\mathbb{C})=\text{GL}_2(\mathbb{Q})^+\left\backslash \left(\left(\mathcal{H}\bigsqcup\mathbb{P}^1(\mathbb{Q})\right )\times \text{GL}_2(\mathbb{A}_f)\right/\mathfrak{U}\right),\]
and $\tau$ is identified with $[\tau,1]$.
For these norm-one units, Shimura reciprocity
\cite[Theorems 6.31 and 6.38]{Shimura}, used as in
\cite[proof of Theorem 3.2]{Yin8} and
\cite[proof of Theorem 4.1]{Yin}, reads
\begin{equation}\label{eq:CM-reciprocity}
[p\tau_r,1]^{\sigma_\ell}
 =[p\tau_r,(D_p)_\ell]\qquad(\ell=3,q),
\end{equation}
where \((D_p)_\ell\in \text{GL}_2(\mathbb{A}_f)\) is \(D_p\) at \(\ell\) and \(1\)
elsewhere. We need only the following two actions.

At \(3\), \(AD\in\Gamma_0(N)\) gives
\[
A^pD_p=h(AD)h^{-1}\in U_3,\qquad
A^p\in U_\ell\quad(\ell\ne3).
\]
Left multiplication by \(A^p\) in
\eqref{eq:CM-reciprocity} therefore gives
\[
P^{\sigma_3}=\varphi^c(A^p(p\tau_r))=[\omega^p]P.
\]
At \(q\), the matrix \(D_p\) is integral, its lower-left
entry is divisible by \(q\), and
\(d(D_p)=r-1\equiv\omega^2\pmod\varpi\).
Thus \eqref{eq:diamond-explicit} gives
\[
P^{\sigma_q}=[\varepsilon^c(r-1)^{-1}]P=[\omega]P.
\]

The two computed actions give
\[
P^{\sigma_3\sigma_q}=[\omega^{p+1}]P=P,
\]
since \(p\equiv-1\pmod3\).
Because \(\sigma_3\sigma_q\) generates the order-three group
\(\Gal(H_{3p}L_{(\varpi)}/H_pL_{(\varpi)})\),
we know that \(P\in E_\varpi(H_pL_{(\varpi)})\).
Each conjugate of \(P^{(\infty)}=[\varepsilon^c(p)]P\)
occurs three times in the full orbit sum, giving
\[
P^{(\infty)}+\sum_{j=0}^{p-1}P^{(j)}
 =[3]\Tr_{H_pL_{(\varpi)}/L_{(\varpi)}}P^{(\infty)}
 =[3\varepsilon^c(p)]z_0.
\]


\emph{3. The Hecke relation.}
Since \(p\) is inert in \(K\), the CM theta series
has no ideal of norm \(p\), so \(a_p(f^c)=0\).
The standard weight two Hecke formula
\cite[Section 5.2]{DS} is therefore
\begin{equation}\label{eq:fourier-sum}
p\varepsilon^c(p)f^c(pz)
 +\frac1p\sum_{j=0}^{p-1}f^c\left(\frac{z+j}{p}\right)
 =a_p(f^c)f^c(z)=0.
\end{equation}
Integrating \eqref{eq:fourier-sum} from \(i\infty\)
and using \(\varphi^c(\infty)=O\) gives
\[
[\varepsilon^c(p)]\varphi^c(p\tau_r)
 +\sum_{j=0}^{p-1}\varphi^c\left(\frac{\tau_r+j}{p}\right)=O.
\]
Together with step 2 this proves \eqref{eq:concrete-sum}.
\end{proof}


\section{The non-$\sqrt{-3}$-divisibility of the trace point $z_0$}
\label{sec:frob}

\subsection{The common reduction of the CM conjugates}

\begin{proposition}\label{prop:base}
The CM parameter \(\tau_r\) in \eqref{eq:parameters} satisfies
\begin{equation}\label{eq:base}
P'=\varphi(\tau_r)=(x',y')\notin E_{\bar\varpi}[\sqrt{-3}]\quad
\varphi^c(\tau_r)=(0,\epsilon\varpi/2)
\end{equation}
with $\epsilon=1$ or $-1$. Furthermore,
\begin{equation}
t^2x',\ y'\in K,
\end{equation}

\begin{equation}
y'+\epsilon \bar\varpi/2=\epsilon\varpi u^3,\quad(u\in K^\times).
\label{eq:base-cube}
\end{equation}
\end{proposition}
\begin{proof}

\cite[Theorem 4.2, Theorem 8.6]{Yin}
says,  \(\varphi^c(\tau_r)\) is $\sqrt{-3}$-torsion while $\varphi(\tau_r)$ is of infinite order\footnote{The condition of $r$ in Theorems 4.1, 4.2 and 8.6 of the version 3 of \cite{Yin} is reversed due to the wrong condition of the adelic congruence group $\mathfrak{U}$ (we should define $\mathfrak{U}$ using matrix component $a$ instead of $d$). This will be modified in the future version.}. \cite[Remark 8.5 and Proposition 5.9]{Yin} says
\[
\varphi^c(\tau_r)=-\varphi^c(0)\ne O.
\]
This immediately proves \eqref{eq:base}.

\cite[Proposition 2.1; Theorem 4.1]{Yin}
and transformation laws \cite[Proposition 3.2]{Yin}
give
\[
P'\in E_{\bar\varpi}(K(\sqrt[3]{\varpi})),\qquad
P'^{\sigma_q}=[\omega^2]P',\qquad t^{\sigma_q}=\omega^2 t.
\]
This gives
\(t^2x',y'\in K\), as required.


Finally, the cubic functions $F_\pm$ in
\cite[(8.0.1)--(8.0.2), Proposition 8.1]{Yin} satisfy
\begin{equation}\label{eq:F}
F_\epsilon^3=
\frac{y\circ\varphi+\epsilon\bar\varpi/2}
     {y\circ\varphi^c+\epsilon\varpi/2},
\qquad F_\epsilon\in K(X_\Gamma).
\end{equation}
Since \(\varphi(Az)=[\omega]\varphi(z)\) and the same holds
for \(\varphi^c\), the quotient \(F_\epsilon(Az)/F_\epsilon(z)\)
is a constant cube root of unity. Its value is \(1\) at
\(\infty\), by the nonzero constant term in
\cite[Proposition 7.4]{Yin}. Hence \(F_\epsilon(Az)=F_\epsilon(z)\).
The same CM reciprocity calculation
\cite[proof of Theorem 4.1]{Yin} gives
\[
[\tau_r,1]^{\sigma_3}=[A\tau_r,1].
\]
Since \(\sigma_3\sigma_q\) fixes \(L_{(\varpi)}\), we obtain
\[
[\tau_r,1]^{\sigma_q}=[A^{-1}\tau_r,1].
\]
Thus \(u=F_\epsilon(\tau_r)\) belongs to \(L_{(\varpi)}\) and
\[
u^{\sigma_q}
=F_\epsilon(A^{-1}\tau_r)
=F_\epsilon(\tau_r)=u.
\]
As \(\sigma_q\) generates
\(\operatorname{Gal}(L_{(\varpi)}/K)\), we have \(u\in K\).
Moreover, \(u\ne0\) by \eqref{eq:base}.
Evaluating \eqref{eq:F} now gives
\[
y'+\epsilon\bar\varpi/2=\epsilon\varpi u^3.
\]


\end{proof}

Fix compatible primes above \(p\) throughout, and write
\(\mathfrak P\) for the chosen prime in a field containing
the CM values. Congruences modulo \(\mathfrak P\) mean
equality in its residue field. Only this residue-field
statement is required; we do not assert divisibility
by the rational integer \(p\) in a ramified integer ring.
In this section we will use\ $\widetilde{\cdot}$\ to mean the reduction modulo
$\mathfrak P$.
\begin{proposition}
\label{prop:frob}Coordinatewise \(p\)-th power $Fr_p$ maps
\(\widetilde E_{\bar\varpi}\) to \(\widetilde E_\varpi\).
For the fixed prime \(\mathfrak P\), there is 
$i=0,1,2$ independent of
\(\sigma\in\Gal(H_pL_{(\varpi)}/L_{(\varpi)})\), such that
\begin{equation}\label{eq:common}
\widetilde{P^\sigma}
=[\omega^i]\widetilde{P'}^{Fr_p}.
\end{equation}
When the affine coordinates are integral, this reads
\begin{equation}\label{eq:coordinate-congruence}
x^c(p\tau_r)\equiv\omega^i x(\tau_r)^p,\qquad
y^c(p\tau_r)\equiv y(\tau_r)^p\pmod{\mathfrak P}.
\end{equation}
Here \(x(z),y(z)\) are the coordinate functions of \(\varphi(z)\),
and \(x^c(z),y^c(z)\) are those of \(\varphi^c(z)\).
In all cases \eqref{eq:common} is an equality of projective
points. 
\end{proposition}
\begin{proof}
By \cite[Section 5]{Yin}, \(y,y^c\) are
conjugate modular functions of level \(\Gamma_0(N/3)\).
Regard them as maps
\(\mathcal Y,\mathcal Y^c:X_0(N)_K\to\mathbf P^1_K\).
They extend at \(p\): apply the extension argument in
\cite[proof of Theorem 7.2]{Yin} to the parametrizations,
compose with \(y:E_a\to\mathbf P^1\), and descend through
the degree-three quotient \(X_\Gamma\to X_0(N)\);
here \(p\nmid N\) and \(p\ne3\).

Write \(S_r,S_{p,r}\) for the points of \(X_0(N)\)
represented by \(\tau_r,p\tau_r\).
Since \(p\) is inert in \(K\), the geometric congruence
in \cite[proof of Proposition 2.4]{SSY}, gives
\[
\widetilde S_{p,r}=\widetilde S_r^{Fr_p}.
\]
For coefficients \(a\in\OO_{K,(p)}\), we have
\(a^p\equiv\bar a\pmod p\), since \(\omega^p=\omega^2\).
Conjugating coefficients therefore gives
\[
y(\widetilde P)
 =\widetilde{\mathcal Y^c}(\widetilde S_r^{Fr_p})
 =\widetilde{\mathcal Y}(\widetilde S_r)^p
 =y(\widetilde{P'})^p.
\]
This equality holds in \(\mathbf P^1(\overline{\F}_p)\);
at a pole use the local coordinate \(1/\mathcal Y\).

The points \(\widetilde P\) and
\(\widetilde{P'}^{Fr_p}\) lie on \(\widetilde E_\varpi\)
and have the same \(y\)-coordinate.
If this value is finite, the equation
\(y^2=x^3+\widetilde\varpi^{\,2}/4\) shows that their
\(x\)-coordinates have the same cube.
Thus they differ by \([\omega^i]\).
Finally, \(H_pL_{(\varpi)}/L_{(\varpi)}\) is totally
ramified at \(p\), so its Galois group acts trivially
on the residue field. Hence
\(\widetilde{P^\sigma}=\widetilde P\) for every \(\sigma\),
and the same \(\omega^i\) proves \eqref{eq:common}.
\end{proof}

\subsection{Reduction of the auxiliary point}

Twist the base point $P'$ to
\begin{equation}\label{eq:B}
B=(t^2x',\varpi y')\in E_q(K).
\end{equation}
This follows immediately from Proposition \ref{prop:base}. Moreover,
\[
y(B)+\epsilon q/2=\epsilon\varpi^2u^3.
\]
Since
\((y(B)+q/2)(y(B)-q/2)=x(B)^3\), we obtain
\begin{equation}\label{eq:Bclass}
y(B)+q/2\equiv\varpi^{-\epsilon}
\quad\text{in }K^\times/K^{\times3}.
\end{equation}

\begin{lemma}\label{lem:finite}
For \(m=(p+1)/3\), we have
\begin{equation}\label{eq:Bnonzero}
[\sqrt{-3}m]\widetilde B\ne O.
\end{equation}
\end{lemma}

\begin{proof}
Choose \(\rho\in\F_p^\times\) with \(\rho^3=q\);
this is possible because \(p\equiv2\pmod3\).
The change of variables
\((x,y)\mapsto(x/\rho^2,y/q-1/2)\) identifies \(E_q\)
over \(\F_p\) with \(E_{\min}^1:y^2+y=x^3\).
Thus \cite[(6.6)]{Yin8} gives
\[
E_q(k)\simeq\bigl(\Z/(p+1)\Z\bigr)^2.
\]
Since \(p\nmid p+1\), this is the full
\((p+1)\)-torsion subgroup, we have
\begin{equation}
E_q(k)=E_q[p+1](\bar k).\label{eq:finitegroup}
\end{equation}

Next, \([3m]=-[\sqrt{-3}m]\circ[\sqrt{-3}]\), and \([\sqrt{-3}]\) is
surjective on \(E_q(\bar k)\). Consequently
\[
[\sqrt{-3}]E_q[3m](\bar k)=E_q[\sqrt{-3}m](\bar k).
\]
Together with \eqref{eq:finitegroup}, this proves
\begin{equation}
[\sqrt{-3}]E_q(k)=\ker\bigl([\sqrt{-3}m]:E_q(k)\to E_q(k)\bigr).
\label{eq:imagekernel}
\end{equation}

For \eqref{eq:Bnonzero}, work over the completion \(K_p\),
whose residue field is \(k\).
By \eqref{eq:Bclass}, the cube class of \(y(B)+q/2\)
is represented by the unit \(\varpi^{-\epsilon}\).
Its residue is a noncube by \eqref{eq:aux}.
Hence \(y(B)+q/2\) is not a cube in \(K_p\).

Suppose now that \([\sqrt{-3}m]\widetilde B=O\).
By \eqref{eq:imagekernel}, there is a
\(\widetilde Q\in E_q(k)\) with
\([\sqrt{-3}]\widetilde Q=\widetilde B\).
Since \(E_q\) has good reduction at $p$ and \([\sqrt{-3}]\) has degree
\(3\), prime to \(p\), Hensel's lemma lifts this equality
to \([\sqrt{-3}]Q=B\) for some \(Q\in E_q(K_p)\).
Write \(Q=(X,Y)\); here \(X\ne0\), since \(B\ne O\).
The addition law on \(E_q\) gives
\begin{equation}\label{eq:cube-pullback}
y([\sqrt{-3}](X,Y))+\frac q2
 =\left(\frac{Y/\sqrt{-3}-q/2}{X}\right)^3
 \qquad(X\ne0).
\end{equation}
This makes \(y(B)+q/2\) a cube in \(K_p\),
a contradiction. This proves \eqref{eq:Bnonzero}.
\end{proof}

The twist \((x,y)\mapsto(t^2x,\varpi y)\) extends over the
local integers at \(p\), since \(t\) is a unit, and commutes
with \([\sqrt{-3}]\).
Transporting \eqref{eq:Bnonzero} to \(\kappa=\F_{p^6}\)
therefore gives
\begin{equation}\label{eq:P0nonzero}
[\sqrt{-3}m]\widetilde{P'}\ne O.
\end{equation}
The group calculation was made over \(k\), before this
transport to \(\kappa\).

\subsection{$z_0$ is not divisible by \texorpdfstring{\(\sqrt{-3}\)}{sqrt(-3)}}

\begin{proposition}\label{prop:z0}
The ordinary trace \(z_0\) satisfies
\begin{equation}\label{eq:z0-three}
[3]z_0=O,\qquad [\sqrt{-3}]z_0\ne O,\qquad
z_0\notin[\sqrt{-3}]E_\varpi(L_{(\varpi)}).
\end{equation}
\end{proposition}
\begin{proof}
The first assertion is \eqref{eq:concrete-sum}.
There are \(m\) terms in the ordinary trace.
By \eqref{eq:common},
\begin{equation}
\widetilde z_0=[\omega^i][m]\widetilde{P'}^{Fr_p}.
 \label{eq:z0reduction}
 \end{equation}
Since coordinate Frobenius commutes with the group law and sends \(\omega\) to \(\omega^2\), $[\sqrt{-3}]\circ Fr_p=Fr_p\circ [-\sqrt{-3}]$. So
 \begin{equation}
[\sqrt{-3}]\widetilde z_0
 =-[\omega^i]
       \bigl([\sqrt{-3}m]\widetilde{P'}\bigr)^{Fr_p}\ne O.
\label{eq:z0detector}
\end{equation}
The last inequality follows from \eqref{eq:P0nonzero},
injectivity of Frobenius, and the fact that multiplication
by \(\omega\) is an automorphism.
Thus \(z_0\) is three-torsion
outside the \(\sqrt{-3}\)-kernel.

The three-division polynomial on \(E_\varpi\) is
\(3x(x^3+\varpi^2)\).
The six points of \(E_\varpi[3]\setminus E_\varpi[\sqrt{-3}]\)
are therefore
\[
(-\omega^i t^2,\ \pm\varpi \sqrt{-3}/2),\qquad i=0,1,2.
\]
At such a point, \(y+\varpi/2\) equals
\(-\varpi\omega^2\) or \(-\varpi\omega\).
Since \(\varpi=t^3\) and \(-1\) are cubes in \(L_{(\varpi)}\),
its class mod cubes is \(\omega^2\) or \(\omega\).
Neither \(\omega\) nor \(\omega^2\) is a cube in
\(L_{(\varpi)}\).
Otherwise that cubic field would contain a primitive
ninth root of unity and would equal \(K(\zeta_9)\).
But \(L_{(\varpi)}/K\) ramifies at \((\varpi)\),
whereas \(K(\zeta_9)/K\) does not.
If \(z_0=[\sqrt{-3}](X,Y)\) for a point in \(E_\varpi(L_{(\varpi)})\),
then \(X\ne0\), since the origin and the points with \(X=0\)
are sent to \(O\). Formula \eqref{eq:cube-pullback}, now
with replacing $q$ by \(\varpi\), would make \(y(z_0)+\varpi/2\) a cube.
This contradiction proves the last assertion of
\eqref{eq:z0-three}.
\end{proof}

\section{Proof of the main result}
\label{sec:non-torsion}

Let $z_1, z_2$ be the points in \eqref{eq:z1}, \eqref{eq:z2}.
\begin{proposition}\label{lem:three}
Both \(z_1\) and \(z_2\) are nonzero.
\end{proposition}
\begin{proof}
Use actual points on \(E_\varpi\):
\[
w_1=-[\omega]z^{\sigma_p}
       +[\omega^2]z^{\sigma_p^2},\qquad
w_2=[\omega^2]z^{\sigma_p}
       -[\omega]z^{\sigma_p^2}.
\]
The relations \(\omega^2+\omega+1=0\),
\(\sqrt{-3}\omega=\omega^2-1\), and \(\sqrt{-3}\omega^2=1-\omega\)
give
\begin{align}
[\sqrt{-3}]w_1&=z_0-z_1,
&
w_1^{\sigma_p}-w_1&=[\omega^2]z_1,
\label{eq:w1}\\
[\sqrt{-3}]w_2&=z_0-z_2,
&
w_2^{\sigma_p}-w_2&=-[\omega]z_2.
\label{eq:w2}
\end{align}
For example,
\[
[\sqrt{-3}]w_1=[1-\omega^2]z^{\sigma_p}
        +[1-\omega]z^{\sigma_p^2}=z_0-z_1,
\]
and
\[
w_1^{\sigma_p}-w_1
=[\omega^2]z+[\omega]z^{\sigma_p}+z^{\sigma_p^2}
=[\omega^2]z_1.
\]
The second pair follows by the same coefficient calculation.

If \(z_1=O\), equations \eqref{eq:w1} make \(w_1\)
fixed by \(\sigma_p\). Since \(w_1\) is defined over
\(L_{(\varpi,p)}\), it belongs to \(E_\varpi(L_{(\varpi)})\).
The other identity would give \([\sqrt{-3}]w_1=z_0\),
contradicting Proposition~\ref{prop:z0}.
The equations \eqref{eq:w2} give the identical contradiction
if \(z_2=O\). Thus both points are nonzero.
\end{proof}


\begin{lemma}\label{lem:formal}
Let an elliptic curve have good reduction over a finite
extension of \(\Q_p\) of absolute ramification index
\(e<p-1\). The kernel of reduction has no nonzero torsion.
\end{lemma}
\begin{proof}
This is the consequence of \cite[Proposition VII.2.2, Proposition IV.3.2 and Theorem IV.6.1]{Silverman}.
%
\end{proof}

\begin{proof}[Proof of Theorem \ref{thm:main}]
The extension \(L_{(\varpi,p)}/L_{(\varpi)}\)
is totally ramified at \(p\), so the three points
\(z,z^{\sigma_p},z^{\sigma_p^2}\) have the same reduction.
Equations \eqref{eq:z1}--\eqref{eq:z2} imply
\[
\widetilde z_1=\widetilde z_2=O,
\]
since \(1+\omega+\omega^2=0\).
The curve \(E_\varpi\) has good reduction at \(p\),
and \eqref{eq:e} gives \(e=3<p-1\), as \(p\ge17\).
Lemma~\ref{lem:formal} applies over the completion of
\(L_{(\varpi,p)}\).
Both \(z_1\) and \(z_2\) are nonzero by Proposition~\ref{lem:three},
so each has infinite order.
\end{proof}

\section{The example \texorpdfstring{\(p=17,\ q=13\)}{p=17, q=13}}
\label{sec:example}

Use the example in \cite[Section 10, example (2)]{Yin}:
\[
\varpi=4+3\omega,\quad\bar\varpi=1-3\omega,\quad
t^3=\varpi,\quad r=23,\quad
\tau_r=-\frac1{3(\omega+23)}.
\]
Indeed \(r^2-r+1=507=13\cdot39\).
The normalized images are
\[
P'=\left((-7-2\omega)t^{-2},\frac{7+9\omega}{2}\right),
\qquad \varphi^c(\tau_r)=(0,\varpi/2),
\qquad F_+(\tau_r)=1.
\]
Thus \(\epsilon=1,u=1\) in \eqref{eq:base-cube}, and
\begin{equation}\label{eq:example-B}
B=\left(-7-2\omega,\frac{1+30\omega}{2}\right)\in E_{13}(K),
\qquad y(B)+13/2=\varpi^2.
\end{equation}

Over \(k=\F_{17}[\omega]/(\omega^2+\omega+1)\).
The reduction of \(E_{13}\) is \(y^2=x^3+4\), and
\[
\widetilde B=(10+15\omega,\ 9+15\omega).
\]
then 
\[
[6]\widetilde B=(16+16\omega,\ 2+4\omega).
\]
Now \(\sqrt{-3}=\omega(1-\omega)\), so
\[
[\sqrt{-3}][6]\widetilde B
=[\omega]\bigl([6]\widetilde B-[\omega][6]\widetilde B\bigr)
=(0,15)=(0,-2)\ne O.
\]
This is the concrete nonzero point in
\eqref{eq:Bnonzero}. Here \([H_{17}:K]=6\),.

\begin{equation}\label{eq:P}
  P=\varphi^c(17\tau_r)
   =\left(\frac{t^2\mathcal{A}(\sqrt[6]{17})}{6},\frac{\varpi \mathcal{B}(\sqrt[6]{17})}{6}\right)
   \in E_\varpi(H_{17}K(t)).
\end{equation}
where
\begin{align}
  \mathcal{A}(T)={}&43+55\omega+(32+4\omega)T+(5-16\omega)T^2\notag\\
         &-(11+13\omega)T^3-(8+2\omega)T^4+(-1+4\omega)T^5,
         \label{eq:A}\\
  \mathcal{B}(T)={}&-93-377\omega-(237+177\omega)T+(-111+35\omega)T^2\notag\\
         &+(24+93\omega)T^3+(57+43\omega)T^4+(27-9\omega)T^5.
         \label{eq:C}
\end{align}
The final traces are
\begin{align}
  {z_0=\left(-t^2,\frac{\varpi \sqrt{-3}}{2}\right),}\\
  {z_1=\left(-\frac{7t^2}{\sqrt[3]{p}^2},-\frac{19\varpi \sqrt{-3}}{34}\right),}\\
  {z_2=\left(-\frac{31t^2}{\sqrt[3]{p}^4},\frac{109\varpi \sqrt{-3}}{578}\right).}
\end{align}
In particular, $z_0\in E_\varpi[3]\setminus E_\varpi[\sqrt{-3}]$.

\bibliographystyle{plain}

\end{document}